\documentclass[12pt,oneside]{amsart}
\usepackage[foot]{amsaddr}
\usepackage[utf8]{inputenc}
\usepackage[T1]{fontenc}
\usepackage{lmodern}

\usepackage{amssymb}
\usepackage{geometry}
\usepackage{amsmath}
\usepackage{amsthm}
\usepackage{graphicx}
\usepackage{color}
\usepackage[normalem]{ulem}

\usepackage[onehalfspacing]{setspace}
\usepackage{verbatim}
\usepackage[hidelinks]{hyperref}

\newcommand{\IR}{\ensuremath{\mathbb{R}}}
\newcommand{\IN}{\ensuremath{\mathbb{N}}}
\newcommand{\IZ}{\ensuremath{\mathbb{Z}}}

\newcommand{\IC}{\ensuremath{\mathbb{C}}}

\newcommand{\IP}{\ensuremath{\mathbb{P}}}
\newcommand{\IE}{\ensuremath{\mathbb{E}}}

\newcommand{\eps}{\varepsilon}
\newcommand{\wtF}{ F_0}
\renewcommand{\rho}{\varrho}

\newcommand{\norm}[1]{\left\Vert#1\right\Vert}
\newcommand{\set}[1]{\left\{#1\right\}}

\newcommand{\brackets}[1]{\left(#1\right)}

\newcommand{\scalar}[2]{\left\langle#1,#2\right\rangle}

\newcommand{\diag}{\mathop{\mathrm{diag}}}
\DeclareMathOperator{\rank}{rank}

\newtheorem{thm}{Theorem}

\newtheorem*{open}{Open Problem}
\theoremstyle{plain}
\newtheorem{lemma}[thm]{Lemma}
\newtheorem{prop}[thm]{Proposition}
\theoremstyle{definition}
\newtheorem{ex}{Example}
\newtheorem{rem}[thm]{Remark}

\title[Function values are enough -- Part II]{Function values are enough\\ 
		for $L_2$-approximation: \\ Part II}

\author{
David Krieg$^1$
\and
Mario Ullrich$^{1,2}$
}

\address{$^1$Institut f\"ur Analysis, 
Johannes Kepler Universit\"at Linz, Austria}
\address{$^2$Moscow Center for Fundamental and Applied Mathematics, 
Lomonosov Moscow State University, Moscow, Russia}
\email{
david.krieg@jku.at, 
mario.ullrich@jku.at}

\keywords{$L_2$-approximation, 
information-based complexity,
least squares, 
rate of convergence, 
random matrices, 
Kadison-Singer}
\subjclass[2010]{%
41A25, %Rate of convergence, degree of approximation
41A45, %Approximation by arbitrary linear expressions
41A65, %Abstract approximation theory (approximation in normed linear spaces and other abstract spaces)
60B20, %Random matrices
41A63. %Multidimensional problems
}

\date{\today}

\begin{document}

\begin{abstract}
In the first part we have shown that,
for $L_2$-approximation 
of functions from a separable Hilbert space
in the worst-case setting,
linear algorithms based on function values
are almost as powerful
as arbitrary linear algorithms 
if the linear widths are square-summable. 
That is, they achieve the same polynomial rate of convergence.
In this sequel, we prove a similar result 
for separable Banach spaces
and other classes of functions.
\end{abstract}

\maketitle

%\newpage

Let $F$ be a set of %real- or 
complex-valued functions on a set $D$
such that, for all $x\in D$, point evaluation
$$
 \delta_x\colon F \to \IC,\quad f\mapsto f(x)
$$
is continuous with respect to some metric $d_{F}$ on $F$.
We consider numerical approximation of functions from such classes, 
using only function values, and
measure the error in the space $L_2=L_2(D,\mathcal{A},\mu)$
of square-integrable functions with respect to an arbitrary 
measure $\mu$ such that $F$ is embedded into $L_2$.
We are interested in the \emph{$n$-th minimal worst-case error}
\begin{equation}
\label{def:en}
%e_n \,:=\, 
e_n(F,L_2) \,:=\, 
\inf_{\substack{x_1,\dots,x_n\in D\\ \varphi_1,\dots,\varphi_n\in L_2}}\, 
\sup_{f\in F}\, 
\Big\|f - \sum_{i=1}^n f(x_i)\, \varphi_i\Big\|_{L_2},
\end{equation}
which is the worst-case error of an optimal linear algorithm that  
uses at most $n$ function values.
These numbers are sometimes called \emph{sampling widths} of $F$. 
We want to compare the sampling widths 
with the \emph{linear widths} of $F$,
defined by
\begin{equation}
\label{def:an}
%a_n \,:=\, 
a_n(F,L_2) \,:=\,
\inf_{\substack{T\colon L_2 \to L_2\\ \rank(T) \,\le\, n}}\, 
\sup_{f\in F}\, 
\big\|f - Tf \big\|_{L_2}.
% \Big\|f - \sum_{i=1}^n \big\langle f, g_i \big\rangle_{L_2} \, \varphi_i\Big\|_{L_2}.
\end{equation}
This is
the worst-case error of an optimal linear algorithm that uses 
at most $n$ linear functionals as information.
In other words, we want to compare the power of function values
(also called standard information)
with the power of arbitrary linear information 
for $L_2$-approximation with linear algorithms.
\medskip

The numbers $e_n$ and $a_n$ are well studied for 
many particular classes of functions.
For an exposition of known results 
and history on these and related quantities,
we refer to the books 
\cite{NW08,NW10,NW12}, especially~\cite[Chapter~26~\&~29]{NW12}, 
as well as \cite{DTU16,Tem93,Tem18}
and references therein. 
Note that the linear widths coincide with the Kolmogorov widths 
in our setting.
Here, we want to relate $e_n$ and $a_n$ for general function classes~$F$.

\medskip

Clearly, we always have $a_n \le e_n$.
On the other hand, an example 
of Hinrichs, Novak and Vybiral from~\cite{HNV08} 
shows that it is not possible to give any general upper bound
for the sampling widths in terms of the linear widths
if the latter are not square-summable.
We therefore ask for such a relation in the case
that the linear widths are square-summable.
Let us formulate 
one particularly interesting open question.

\begin{open}
Is it true for any class $F$ and any measure $\mu$ as above
that there exists a constant $c \in \IN$ 
(possibly depending on $F$ and $\mu$) such that
\begin{equation}\label{OP}
e_{cn}(F,L_2) \,\le\, \sqrt{ \frac{\vphantom{1}c}{n} \sum_{k\ge n} a_k(F,L_2)^2}
\quad\text{ for all } n\in\IN \text{ ?} 
\end{equation}
\end{open}

Note that \eqref{OP} is of no use
if the linear widths are not square-summable, 
but would lead to a quite tight bound otherwise.
The relation \eqref{OP} is true
for all examples of sufficiently studied function classes $F$
that are known to the authors. 
On the other hand, general results are only known in the case that $F$
is the unit ball of a reproducing kernel Hilbert space.
Even in this case, the problem is open, but the gap is quite small already.
Namely, it was shown in \cite{KU19} that \eqref{OP} is true
with the index $cn$ replaced by $cn\log(n+1)$, 
showing that the polynomial order of approximation and sampling widths 
is the same if the linear widths are square-summable. 
This was then improved by Nagel/Sch\"afer/T.\,Ullrich in~\cite{NSU20},
who showed \eqref{OP} with an additional factor of $\sqrt{\log(n+1)}$ 
on the right hand side.
In all these papers, the constant $c$ is independent of $F$ and $\mu$.
The aim of this paper is to prove a similar result 
for general function classes $F$.
\medskip

Before we come to our result, let us mention that, 
recently, another beautiful 
upper bound for the sampling widths of general classes $F$ 
has been shown by Temlyakov in~\cite{Tem20}.
This will be discussed in Section~\ref{sec:Tem}.
\medskip

Our main result reads as follows.

\begin{thm}\label{thm:main}
 Let $(D,\mathcal A,\mu)$ be a measure space and let 
 $F$ be a separable metric space of complex-valued functions on $D$
 that is continuously embedded into $L_2(D,\mathcal A,\mu)$
 such that function evaluation is continuous on $F$.
 Assume that $(a_n(F, L_2))\in\ell_p$ for some $0<p<2$.
 There is a universal constant $c\in \IN$
 and a constant $c_p>0$, depending only on $p$, such that, for all $n\ge 2$, we have
 \[
 e_{cn}(F, L_2) \,\le\, 
 c_p\, \sqrt{\log n}\, \left(\frac1n \sum_{k\ge n} a_k(F, L_2)^p \right)^{1/p}.
 \]
\end{thm}

\medskip

 In particular, 
 the linear widths and the sampling widths 
 have the same polynomial order of convergence:
 If we assume that $a_n(F, L_2) \lesssim n^{-\alpha} \log^\beta (n+1)$
 for some $\alpha>1/2$ and $\beta \in \IR$, 
 then we obtain
\begin{equation}\label{eq:order}
  e_n(F, L_2) \,\lesssim\,  n^{-\alpha} \log^{\beta+1/2} (n+1).
\end{equation}
Here and in the following, the symbol $\lesssim$ means that the left hand side
is bounded by a constant multiple of the right hand side for all $n\in\IN$,
whereas the symbol $\asymp$ means that this relation holds in both directions.
We will also present a bound for the case 
$\alpha=1/2$ and $\beta<-3/2$ which is off by an additional log-factor, 
see Section~\ref{sec:limit-case}. 
It is still open if the result of Theorem~\ref{thm:main} also holds 
for $p=2$, as it does for Hilbert spaces, see~\cite{NSU20}.

\medskip

Our proof of Theorem~\ref{thm:main} is not constructive.
However, if we know the operators $T$ 
that achieve the infimum in \eqref{def:an},
possibly up to a multiplicative constant,
then we can provide an explicit weighted least squares estimator
that achieves the stated upper bound up to a further logarithmic factor with high probability,
see Theorem~\ref{thm:main3}.
Note that these operators are known for a huge variety
of smoothness spaces, e.g., on the $d$-dimensional torus, see~\cite{DTU16}.

\section{Discussion and examples}

In this section, we give some additional comments on our 
main result, together with a few illustrative examples 
and a comparison with existing results.

\subsection{The condition on $F$}

The natural condition appearing in the proof of Theorem~\ref{thm:main} 
is that $F$ is a countable subset of $L_2$, see Theorem~\ref{thm:main2} 
for a precise statement.
We then need some kind of continuity
in order to extend our result to uncountable sets $F$.
Here, we employ that $F$ is
a separable metric space with continuous function evaluation 
and continuous embedding in $L_2$.
These assumptions are satisfied, for example, if
\begin{itemize}
 \item $F$ is the unit ball of a separable 
 normed space on which function evaluation at each point is 
	a continuous functional, or
 \item the measure $\mu$ is finite and $F$ is a compact subset of 
 the space of bounded functions on $D$.
\end{itemize}
To see that the conditions of Theorem~\ref{thm:main} are matched,
we equip $F$ with the metric induced by the normed space in the first case
and with the supremum metric in the second case.
Note that
the theorem might also be applied 
if $F$ is the unit ball of a \emph{non-separable}
normed space, since we might have separability with respect 
to a weaker norm.
For instance, the unit ball $F$ of the non-separable Sobolev space 
$W_\infty^s(0,1)$ with smoothness $s\in\IN$ clearly satisfies
the conditions of Theorem~\ref{thm:main}
when equipped with the supremum metric.

\subsection{Temlyakov's $L_\infty$-bound}
\label{sec:Tem}

Let us shortly compare
our result with
the recent result of Temlyakov~\cite{Tem20}
(see also \cite{CM17} for a related result).
He proved 
that there are universal constants $c,C\in\IN$
such that for any compact domain $D\subset \IR^d$, any probability
measure $\mu$ on $D$, and any compact subset $F$ of
the space $\mathcal C(D)$ of continuous functions on $D$,
we have
\begin{equation}\label{eq:Tem}
 e_{cn}(F,L_2) \,\le\, C\, d_n(F,L_\infty).
\end{equation}
Here, $d_n(F,L_\infty)$ is the $n$th Kolmogorov width of $F$ in $L_\infty$.
\medskip

First, we observe that the assumptions of 
Theorem~\ref{thm:main} and \eqref{eq:Tem} are quite different.
The result
\eqref{eq:Tem}
does not require the square-summability 
of the linear widths. For example, it can be applied
for classes of functions with small mixed smoothness 
as considered in \cite{TU21a,TU21b},
where Theorem~\ref{thm:main} fails to be of use.
On the other hand, only Theorem~\ref{thm:main} 
can be applied for classes of
unbounded functions on unbounded domains, 
like
Hermite spaces on $\IR^d$
or spaces of functions with singularities. 
But also for
classes of bounded functions on compact domains,  
it is not possible to say that one result yields better estimates than the other.
This is illustrated by the following example, 
which was kindly provided to us by Erich Novak.

\begin{ex}\label{counterexample}
Let $\mu$ be the Lebesgue measure on $[0,1]$ and 
let $(\ell_i)_{i\in\IN}$ and $(h_i)_{i\in\IN}$ be decreasing 
zero-sequences with $\sum_{i=1}^\infty \ell_i=1$.
Let $b_i$ be the hat function with height one, supported
on the interval $I_i=[\sum_{k=1}^{i-1} \ell_k, \sum_{k=1}^{i} \ell_k]$
of length $\ell_i$. We consider 
\[
F=\Big\{ \sum_{i=1}^\infty \lambda_i b_i \,\big\vert\, |\lambda_i| \le h_i \text{ for all } i \Big\}.
\]
Note that $F$ is a compact subset of $\mathcal C([0,1])$
which follows from the theorem of Arzel\`a  and Ascoli.
For this example, one can compute that
\[
 e_n(F,L_2) \, = \, \left( \int_0^1 \Big( \sum_{i>n} h_i b_i(x) \Big)^2~{\rm d} x \right)^{1/2}
\,=\, \left( \frac13 \sum_{i>n} h_i^2 \ell_i \right)^{1/2}
\]
and that the Kolmogorov widths in the uniform norm are given by
\[
 d_n(F,L_\infty) \,=\, h_{n+1}. %\mathcal C([0,1])
\]
By choosing $\ell_i=i^{-\alpha}/\sum_{k\in\IN} k^{-\alpha}$ with some $\alpha>1$
and $h_i=i^{-\beta}$ with some $\beta>0$,
we obtain that the sampling widths are of order $n^{-(\beta+(\alpha-1)/2)}$
while the Kolmogorov widths are of order $n^{-\beta}$.
If $\alpha$ is close to $1$ and $\beta<1/2$, then \cite{Tem20}
yields an almost optimal bound while our result yields nothing.
If $\alpha>2$ and $\beta$ is close to zero,
then our results yields an almost optimal bound, 
while \cite{Tem20} yields almost nothing.
\end{ex} 

In contrast to this example, it is quite 
remarkable that Theorem~\ref{thm:main} 
and~\eqref{eq:Tem} actually lead to the same 
upper bounds for many multivariate function classes 
of mixed smoothness as considered in~\cite{DTU16}.
We note that the optimal order of the numbers $d_n(F,L_\infty)$ 
is often not known, in contrast to the linear widths $a_n(F, L_2)$.

\subsection{Korobov classes} 
As a further example, let us consider the Korobov classes
\[
 E^r_d\,=\,\left\{ f\in L_1(\mathbb T^d) \ \big\vert\ |\hat f(\mathbf k)| \le \prod_{j=1}^d \max\{1, |k_j|\}^{-r} \text{ for all } \mathbf k \in \IZ^d \right\},
\]
for given $r>1$ and $d\in\IN$, see e.g.\ \cite[Section~3.3]{DTU16}.
Here, $\mathbb T^d$ is the $d$-dimensional torus and
$\hat f(\bf k)$ denotes the (classical) Fourier coefficient.
It is well known that the non-increasing rearrangement $(c_n)_{n\in\IN}$
of the sequence $(\prod_{j=1}^d \max\{1, |k_j|\}^{-r})_{\mathbf k \in \IZ^d}$
satisfies $c_n \asymp n^{-r}\, \log^{r(d-1)} (n+1)$.
This can be derived from~\cite{Bab60,Mit62}, 
see also \cite{Kri18} for a direct formulation.
We easily obtain
\[
 a_n(E^r_d,L_2) \,\le\, \Big( \sum_{k>n} c_k^2 \Big)^{1/2}  
\,\lesssim\, n^{-r+1/2}\, \log^{r(d-1)} (n+1)
\]
and it follows from Lemma~3.4.5 and Theorem~4.3.5 in~\cite{DTU16} that
this bound is optimal.
Moreover, 
%the class 
$E^r_d$ is a bounded subset of 
the \emph{mixed smoothness} Sobolev space $\mathbf W_2^s(\mathbb T^d)$
for all $s< r-1/2$ and therefore a compact subset of $\mathcal C(\mathbb T^d)$.
Theorem~\ref{thm:main}
%, see~\eqref{eq:order}, 
yields 
\begin{equation*}%\label{eq:Korobov}
 e_n(E^r_d,L_2) \,\lesssim\, n^{-r+1/2}\, \log^{r(d-1)+1/2} (n+1).
\end{equation*}
To the best of our knowledge, this bound is new.

\section{The result behind Theorem~\ref{thm:main}}

Our main result is based on the following apparently 
more general theorem. 

\begin{thm}\label{thm:main2}
Let $(D,\mathcal A,\mu)$ be a measure space and let 
$F_0$ 
be a countable 
set of functions 
in $L_2(D,\mathcal A,\mu)$.
Assume that $(a_n(F_0, L_2))\in\ell_p$ for some $0<p<2$.
Then there is a universal constant $c\in \IN$
and a constant $c_p>0$, depending only on $p$, such that,
for all $n\ge 2$, we have
 \[
 e_{cn}(F_0, L_2) 
 \,\le\, c_p\, \sqrt{\log n}\, \left(\frac1n \sum_{k\ge n} a_k(F_0, L_2)^p \right)^{1/p}.
 \]
\end{thm}

Before we prove this theorem, let us show 
how it implies Theorem~\ref{thm:main}.

\begin{proof}[Proof of Theorem~\ref{thm:main}]
Since $F$ is a separable metric space, 
it contains a countable dense subset $ F_0$.
Now, 
let $x_1,\dots,x_n\in D$ and $\varphi_1,\dots,\varphi_n\in L_2$ be arbitrary.
We obtain for every $f\in F$ and $g\in F_0$ that
\[
\Big\|f - \sum_{i=1}^n f(x_i) \varphi_i\Big\|_{L_2}
\,\le\, \Big\|f - g\Big\|_{L_2} 
\,+\, \Big\|g - \sum_{i=1}^n g(x_i) \varphi_i\Big\|_{L_2}
\,+\, \Big\|\sum_{i=1}^n \Bigl(f(x_i)-g(x_i)\Bigr) \varphi_i\Big\|_{L_2}.
\]
To bound 
the first and the last term, first note that 
$U_\delta(f)\cap F_0\neq\varnothing$ 
for every $\delta>0$, where 
\[
U_\delta(f):=\{g\in F\colon d_F(f,g)<\delta\}.
\]
The continuity of the embedding into $L_2$ and of function evaluation 
now implies that for any $\eps>0$ 
there is some $\delta>0$ such 
that $\|f-g\|_{L_2}<\eps$ and 
$|f(y_i)-g(y_i)|<\eps$ for all $i=1,\dots,n$ and all $g\in U_\delta(f)$. 
Therefore, for every $\eps>0$ and every $f\in F$, 
we have 
\[
\Big\|f - \sum_{i=1}^n f(x_i)\, \varphi_i\Big\|_{L_2}
\,<\, \eps \,+\, 
\sup_{g\in F_0}\, \Big\|g - \sum_{i=1}^n g(x_i) \varphi_i\Big\|_{L_2}
\,+\, \eps \sum_{i=1}^n \|\varphi_i\|_{L_2}.
\]
We obtain that
\[
\sup_{f\in F}\, \Big\|f - \sum_{i=1}^n f(x_i)\, \varphi_i\Big\|_{L_2}
\,=\, 
\sup_{g\in F_0}\, \Big\|g - \sum_{i=1}^n g(x_i)\, \varphi_i\Big\|_{L_2}
\]
for all $x_1,\hdots,x_n\in D$ and $\varphi_1,\hdots,\varphi_n\in L_2$.
Therefore, an error bound of an algorithm on $F_0$ 
carries over to $F$ and so, together with $a_k(F_0, L_2)\le a_k(F, L_2)$
Theorem~\ref{thm:main2} implies Theorem~\ref{thm:main}.
\end{proof}

We will now prove Theorem~\ref{thm:main2}
by proving an error bound for a specific algorithm on $F_0$.
Recall that we have just proven that the same algorithm works for 
the class~$F$ from Theorem~\ref{thm:main}
if we choose $F_0$ as a countable dense subset.

\section{Algorithm and Proof of Theorem~\ref{thm:main2}}

In this whole section, we work under the assumptions of 
Theorem~\ref{thm:main2}.
We start with the following simple observation.

\begin{lemma}\label{lem:basis}
There is an orthonormal system 
$\{b_k \colon k\in \IN \}$ 
in $L_2$ such that
the orthogonal projection $P_n$ onto the span 
$V_n={\rm span}\{b_1,\hdots,b_n\}$ satisfies
\begin{equation}
\label{eq:projections}
 \sup_{f\in F_0} \Vert f - P_n f \Vert_{L_2}  
\,\le\, 2\, a_{n/4}(F_0,L_2),
 \qquad n\in \IN.
\end{equation}
\end{lemma}

Note that the definition of $a_n$ in \eqref{def:an} 
makes perfect sense for $n\notin\IN$.

\begin{proof}
Clearly it is enough to find an increasing sequence of 
subspaces of $L_2$,
 \[
  U_1 \subseteq U_2 \subseteq U_3 \subseteq \hdots,
  \qquad
  \dim(U_n) \le n,
 \]
 such that the projection $P_n$ onto $U_n$ satisfies \eqref{eq:projections}.
 By the definition of $a_m$, $m\in\IN$, 
 there is a subspace $W_m\subset L_2$ of dimension~$m$
and a linear operator $T_m\colon L_2 \to W_m$
 such that 
 \[
  \sup_{f\in F_0} \Vert f - T_m f \Vert_{L_2} \,\le\,  2\, a_m(F_0,L_2).
 \]
 We let $U_n$ be the space that is spanned by the union of the 
 spaces $W_{2^k}$ over all $k\in \IN_0$ with $2^k\le n/2$.  
 Note that $U_n$ contains a subspace 
 $W_m$ with $m\ge n/4$.
 Therefore, 
 $P_n f$ is at least as close to $f$ as $T_m f$ for some $m\ge n/4$,
 which implies~\eqref{eq:projections}.
\end{proof}

\smallskip

In what follows 
$\{b_k \colon k\in \IN \}$ will always be
the orthonormal system 
from Lemma~\ref{lem:basis}. 
Note that we will consider $b_k$ as a function,
where we take an arbitrary representer from the equivalence class in $L_2$.
We will denote
\[
 \varepsilon_n \,:=\, \sup_{f\in F_0} \Vert f- P_n f \Vert_{L_2}
\]
to ease the notation, keeping in mind that 
$\varepsilon_n \le 2\, a_{n/4}(F_0, L_2)$.
We have almost sure convergence of the (abstract) 
Fourier series on $F_0$.

\begin{lemma}\label{lem:almost-sure}
There is a measurable subset $D_0$ of $D$
with $\mu(D\setminus D_0)=0$ such that
for all $x\in D_0$ and $f\in F_0$ we have 
\[
 f(x) \,=\, \sum_{k\in \IN} \hat f(k)\, b_k (x),
 \qquad \text{where} \quad \hat f(k) := \scalar{f}{b_k}_{L_2}.
\]
\end{lemma}

\begin{proof}
For all $f\in F_0$, we have
\begin{equation*}
 \sum_{k\in \IN} |\hat f(k)|^2\, k
 \,=\, \sum_{n\in \IN_0} \sum_{k>n} |\hat f(k)|^2
 \,\le\, \sum_{n\in \IN_0} \varepsilon_n^2 < \infty.
\end{equation*}
The Rademacher-Menchov Theorem, see e.g.~\cite{S41}, now implies that 
the Fourier series of $f$ converges to $f$ almost everywhere.
Since $F_0$ is countable, the almost everywhere convergence holds
simultaneously for all $f\in F_0$.
\end{proof}

%%%%% EVENTUELL IRGENDWANN INTERESSANT:
%We note that, in the case that $F$ is the
%unit ball of a separable normed space
%with continuous function evaluation and continuous embedding in $L_2$,
%the almost everywhere convergence may also be shown on the full space.

\begin{rem}
Note that the Rademacher-Menchov Theorem holds under much weaker assumptions. 
However, this is not needed here, because we anyhow require at least the
square-summability of the $\eps_n$.
\end{rem}

The proof of our bound on the sampling widths is based on 
an error bound for a specific algorithm. 
This is a \emph{weighted least squares estimator}
of the form 
\begin{equation}\label{eq:alg}
A_{m,n}(f) \,=\, 
\underset{g\in V_n}{\rm argmin}\, \sum_{i=1}^m \frac{\vert g(x_i) - f(x_i) \vert^2}{\varrho(x_i)}, 
\end{equation}
for some $x_1,\dots,x_m\in D$ and $m\ge n$. 
We give an explicit formula 
for the weight function $\rho\colon D\to \IR$ later.
The points $x_1,\dots,x_m$ will be
obtained via a probabilistic argument.
They will satisfy $\rho(x_i)>0$.
The algorithm %$A_m=
$A_{m,n}$ %from~\eqref{eq:alg} 
may be written as 
\[
 A_{m,n}\colon \wtF\to L_2, \qquad A_{m,n}(f):=\sum_{k=1}^n (G^+ N f)_k\, b_k
\]
where $N\colon F_0 \to \IR^m$ with 
$N(f):=\left(\varrho(x_i)^{-1/2}f(x_i)\right)_{i\leq m}$ 
is the \emph{weighted information mapping} and
$G^+\in \IR^{n\times m}$ is the Moore-Penrose inverse of the matrix 
\begin{equation}\label{eq:G}
 G := \left(\varrho(x_i)^{-1/2} b_k(x_i)\right)_{i\leq m, k\leq n} 
\in \IR^{m\times n}.
\end{equation}
This description of $A_{m,n}$ is actually more precise 
since it also specifies $A_{m,n}(f)$ in the case that the argmin
in~\eqref{eq:alg} is not unique (which is equivalent to $G$ not
having full rank).
For the state of the art 
on (weighted) least squares methods for the approximation 
of individual functions, or in a randomized setting, 
we refer to~\cite{CDL13,CM17} and references therein.
Here, we consider such methods in the 
\emph{worst-case setting}, i.e., we measure the error via
\[
e(A_{m,n},F_0,L_2) \,:=\, 
\sup_{f\in F_0}\,  \big\|f - A_{m,n}(f)\big\|_{L_2}.
\]
Clearly, we have $e_m(F_0,L_2)\leq e(A_{m,n},F_0,L_2)$ for 
every choice of 
$x_1,\dots,x_m$ 
and $\rho$.

\medskip

The proof of our upper bound uses the following 
simple lemma, see~\cite{KU19}, 
which we prove for the reader's convenience.
Note that our systematic study of the ``power of random information'' 
was initiated in~\cite{HKNPU19a}, see also~\cite{HKNPU19b}, 
and therefore some of the basic ideas behind, 
like a version of the following lemma, 
already appeared there.

\begin{lemma}\label{lem:bound}
Assume that $G$ has full rank, then 
\[
e(A_{m,n}, F_0, L_2)^2 \,\le\, 
\varepsilon_n^2
+ s_{\rm min}(G \colon \ell_2^n \to \ell_2^m)^{-2}
	\, \sup_{f\in \wtF} \norm{N(f- P_n f)}_{\ell_2^m}^2,
\]
where $s_{\rm min}(G \colon \ell_2^n \to \ell_2^m)$ is
the smallest
singular value of the matrix $G$.
\end{lemma}

\begin{proof}
Since $G$ has full rank we obtain
from \eqref{eq:alg} 
that $A_{m,n}$ satisfies $A_{m,n}(f)=f$ for all $f\in V_n$. 
Using Lemma~\ref{lem:basis}, 
we obtain for any $f\in \wtF$ that
\begin{align*}\allowdisplaybreaks
 \norm{f-A_{m,n}(f)}_{L_2}^2 \,&=\,  \norm{f-P_n(f)}_{L_2}^2 + \norm{P_n  f - A_{m,n}(f)}_{L_2}^2 \\
 \,&\le\, \varepsilon_n^2 + \norm{A_{m,n}(f- P_n f)}_{L_2}^2 
 \,=\, \varepsilon_n^2 + \norm{G^+ N(f- P_n f)}_{\ell_2^n}^2 \\
 &\le\, \varepsilon_n^2 +\norm{G^+\colon \ell_2^m \to \ell_2^n}^2 
\cdot \sup_{f\in \wtF} \norm{N(f- P_n f)}_{\ell_2^m}^2.
\end{align*}
It only remains to note that the norm of $G^+$ is the 
inverse of the 
smallest singular value of the matrix $G$.
\end{proof}

We note that
\[
 \varepsilon_n^2 = (\varepsilon_n^p)^{2/p} 
 \le \Big( \frac2n \sum_{k\ge n/2} \varepsilon_k^p \Big)^{2/p},
\]
due to the monotonicity of $(\varepsilon_n)$.
The main result of this paper therefore follows once we prove
\[
s_{\rm min}(G \colon \ell_2^n \to \ell_2^m)^2 \,\gtrsim\, m 
\]
and
\[
\sup_{f\in \wtF} \norm{N(f- P_n f)}_{\ell_2^m}^2 
\,\lesssim\, n\, \log n \; \Big(\frac1n \sum_{k\ge n/2} \varepsilon_k^p\Big)^{2/p}
\]
for an instance of $A_{m,n}$ with $n\le m\le C n$.
Note that the first bound implies that $G$ has full rank.
We divide the proof of this into two parts:
\begin{enumerate}
\item We show these bounds with high probability 
			for $m\asymp n \log n$ i.i.d.~random points 
			%w.r.t.~$\rho$, 
			and then, 
\item based on the famous solution to the Kadison-Singer problem,
			we extract $m\asymp n$ points that fulfill the same bounds.
\end{enumerate}

\subsection{Random points} 

We first observe the result of this paper holds with high probability 
if we allow a logarithmic oversampling. 
For this, we 
now introduce the sampling density $\rho\colon D\to \IR$, 
which will also specify the weights in the algorithm. 
For  $I_\ell:=\{n 2^\ell+1,\dots,n2^{\ell+1}\}$ 
and 
some monotonically decreasing sequence
$(v_\ell)_{\ell\in \IN_0}$ with 
$\sum_{\ell\ge 0} v_\ell^2 = 1$
(to be specified later) 
we define
\[
\varrho(x) := \frac12 \left(
 \frac1n \sum_{k\le n} |b_k(x)|^2  
+  \sum_{\ell\ge0}\, \frac{v_\ell^2}{|I_\ell|} \sum_{k\in I_\ell} |b_k(x)|^2
 \right).
\]
Observe that $\varrho$ is indeed a $\mu$-density. 
Under the assumption $(a_n)\in\ell_p$ with $p<2$
that is considered in Theorem~\ref{thm:main},
we may just use 
$v_\ell\asymp 2^{-\delta\ell}$ for some $0<\delta <1/p-1/2$. 
If we want to get closer to the condition $(a_n)\in\ell_2$,
we need to 
consider sequences $(v_\ell)$ with polynomial decay, see Section~\ref{sec:limit-case}.

\begin{rem}\label{rem:density}
The form of the density $\varrho$ is very much inspired 
by the density invented in~\cite{KU19}, 
which was already applied in
\cite{KUV19,
%MU20,
NSU20,U20}. 
The density in these papers was needed 
to prove the result for Hilbert spaces in greatest generality, 
i.e., for all
sequences %Hilbert spaces with 
$(a_n)\in\ell_2$.
The density used here is different.
It is not clear to us whether one can use the
density from~\cite{KU19}
for arbitrary classes $F$ and prove a result like
Theorem~\ref{thm:main} for all $(a_n)\in\ell_2$.
Presently, we do not know what happens, e.g., in the case that
$a_n\asymp n^{-1/2}\log^{-3/2}(n+1)$, 
see Section~\ref{sec:limit-case}.
\end{rem}

As the result of this part of the proof might be 
of independent interest, we formulate it as a theorem.

\begin{thm}\label{thm:main3}
Let $(D,\mathcal A,\mu)$ be a measure space and let 
$ F_0$ be a countable
set of functions in
$L_2(D,\mathcal A,\mu)$.
Assume that
 $(a_n(F_0, L_2))\in\ell_p$ for some $0<p<2$.
 Then there is a universal constant $C_1>0$
 and a constant $c_p>0$, depending only on $p$, such that
 for all $n\ge 2$ the algorithm
$A_{m,n}$ from \eqref{eq:alg} with 
$m= \left\lceil C_1\,n \log n\right\rceil$
and i.i.d.~random variables 
$x_1,\dots,x_m$ with $\mu$-density $\rho$, 
satisfies
 \[
 e(A_{m,n}, F_0, L_2)  \,\le\, c_p \,\left(\frac1n \sum_{k\ge n/8} a_k(F, L_2)^p \right)^{1/p}
 \]
 with probability at least $1-\frac{5}{n^2}$.
\end{thm}

The proof of this result follows a similar reasoning
as the original proof in \cite{KU19}, 
with the improvements from~\cite{U20} 
that show the result from \cite{KU19} with high probability.
The crucial difference is that we show an upper bound on
the ``norm'' of the information mapping $N$ on the set
\[
F_1 \,:=\, \set{ \sum_{k>n} c_k b_k \,\big|\, \sum_{k>m} c_k^2 \le \varepsilon_m^2 \text{ for all } m \ge n },
\]
related to full approximation spaces,
instead of the smaller set
\[
 F_2 \,:=\, \set{ \sum_{k>n} c_k b_k \,\big|\, \sum_{k>n} \brackets{\frac{c_k}{\varepsilon_k}}^2 \le 1  },
\]
related to Hilbert spaces and considered in \cite{KU19,U20}.
We do this with the help of a dyadic decomposition of the 
index set $\{k\in\IN \mid k>n\}$,
together with suitable bounds on the norms of the 
corresponding random matrices. 
%As in~\cite{U20},
For this, we use again the matrix concentration result 
from~\cite{O10}, see also~\cite{MP06}.  

\begin{prop}[{\cite[Lemma 1]{O10}}]\label{prop:O}
Let $X$ be a random vector in $\IC^k$ with  
$\|X\|_2\le R$ with probability 1, and let $X_1,X_2,\dots$ 
be independent copies of $X$. 
Additionally, let $E:=\IE(X X^*)$ satisfy $\|E\|\le 1$, 
where $\norm{E}$ denotes the spectral norm of~$E$. 
Then, 
\[
\IP\left(\bigg\|\sum_{i=1}^m X_i X_i^* - mE\bigg\|
	\,\ge\, m\cdot t\right)
\,\le\, 4m^2\exp\left(-\frac{m}{16R^2} s_t\right),
\]
where $s_t=t^2$ for $t\le2$, and $s_t=4(t-1)$ for $t>2$.
\end{prop}

Note that \cite[Lemma 1]{O10} wrongly states 
$s_t=\min\{t^2,4t-4\}$, which can easily be corrected 
by looking into the proof.
Let us now prove the norm bounds that we need.
Namely, we prove the existence of constants $C_1$ and
$C_2=C_2(p)$
such that the following holds for $m=\left\lceil C_1\,n\log n\right\rceil$
and i.i.d.\ points $x_1,\hdots,x_m$ with density~$\varrho$.
\medskip

{\bf Fact 1:} \quad 
$\displaystyle \IP\left[\,
s_{\rm min}(G \colon \ell_2^n \to \ell_2^m)^2 \,<\, \frac{m}{2}
\,\right]
\,\le\, \frac{4}{n^2}$.

\smallskip

{\bf Fact 2:} \quad 
$\displaystyle \IP\left[\,\sup\limits_{f\in \wtF} \norm{N(f- P_n f)}_{\ell_2^m}^2 
\,>\, C_2\, n \log n\, \bigg(\frac1n \sum_{k \ge n/2} \varepsilon_k^p  \bigg)^{2/p}
 \right] \,\le\,  \frac{1}{n^2}$.

\medskip

Together with Lemma~\ref{lem:bound}
and $\varepsilon_k \le 2a_{n/4}$
these bounds clearly imply Theorem~\ref{thm:main3}.

\begin{proof}[Proof of Fact~1]
Let 
$X_i:=\varrho(x_i)^{-1/2}(b_1(x_i), \hdots, b_n(x_i))^\top$.
Then we have $\sum_{i=1}^m X_i X_i^* = G^*G$ with $G$ from \eqref{eq:G}.
First observe
\begin{equation}\label{eq:norm-G}
\norm{X_i}_2^2 \,=\, \varrho(x_i)^{-1} \sum_{k\le n} b_k(x_i)^2 
 \,\leq\, 2 n \,=:\, R^2.
\end{equation}
Since $E=\IE(X X^*)=\diag(1, \hdots,1)$ 
we have $\|E\|=1$. 
Therefore, Proposition~\ref{prop:O} with $t=\frac12$ 
and 
$m=\left\lceil C_1 n\log n\right\rceil$
implies
\begin{equation}\label{eq:bound-G}
 \IP\Big(\norm{G^*G - mE} \ge \frac{m}{2}\Big)
 \le \frac{4}{n^2}
\end{equation}
if the constant $C_1>0$ is large enough.
We obtain 
\[
s_{\rm min}(G \colon \ell_2^n \to \ell_2^m)^2
\,=\, s_{\rm min}(G^*G) \,\geq\, s_{\rm min}(mE) - \|G^*G - mE\|
 \,\geq\, m/2
\]
with probability at least $1-\frac4{n^2}$.
\end{proof}

\begin{proof}[Proof of Fact~2]
Note that we almost surely have for all $i=1,\hdots,m$ 
that $\varrho(x_i)$
is positive and finite and $x_i$ is
contained in the set $D_0$ from Lemma~\ref{lem:almost-sure}
such that we have
$f(x_i) \,=\, \sum_{k=1}^\infty \hat{f}(k) b_k(x_i)$
for every $f\in\wtF$.
In this certain event,
each entry of $N(f-P_nf)\in\IR^m$ 
can be written as 
\[
\varrho(x_i)^{-1/2}\, (f-P_nf)(x_i)
\,=\, \sum_{k>n} \hat{f}(k)\, \varrho(x_i)^{-1/2}\, b_k(x_i).
\]
If we now define $I_\ell:=\{n2^\ell+1,\dots,n2^{\ell+1}\}$ and 
the random matrices 
\[
\Gamma_\ell := 
\left(\varrho(x_i)^{-1/2} b_k(x_i)\right)_{i\leq m, k\in I_\ell} 
\in \IR^{m\times n2^\ell}, 
\]
and set $\hat{f}_\ell:=(\hat{f}(k))_{k\in I_\ell}$,
we obtain that 
\begin{multline*}
\norm{N(f- P_n f)}_{\ell_2^m}
\,=\, \norm{\sum_{\ell=0}^\infty \Gamma_\ell \hat{f}_\ell}_{\ell_2^m} \\ %\qquad \text{a.s.},
\,\le\, \sum_{\ell=0}^\infty
\norm{\Gamma_\ell\colon \ell_2(I_\ell)\to\ell_2^m}   \|\hat{f}_\ell\|_{\ell_2(I_\ell)}
\,\le\, \sum_{\ell=0}^\infty 
\norm{\Gamma_\ell\colon \ell_2(I_\ell)\to\ell_2^m}  \varepsilon_{n2^\ell}.
\end{multline*}
It remains to bound the norms of $\Gamma_\ell$ with 
high probability, 
simultaneously for all~$\ell$. 

Let us start with an individual $\ell$, 
and consider 
$X_i:=\varrho(x_i)^{-1/2}(b_k(x_i))_{k\in I_\ell}^\top$
with $x_i$ distributed according to $\varrho$. 
Hence, $\sum_{i=1}^m X_i X_i^* = \Gamma_\ell^*\Gamma_\ell$. 
We see that 
\[\begin{split}
\norm{X_i}_2^2 \,&=\, \varrho(x_i)^{-1} \sum_{k\in I_\ell} |b_k(x_i)|^2
\,\le\, \frac{n2^{\ell+1} \sum_{k\in I_\ell} |b_k(x_i)|^2}{v_\ell^2 \sum_{k\in I_\ell} |b_k(x_i)|^2}\\
\,&\le\, n2^{\ell+1} %\frac12 
v_\ell^{-2}
\,=:\, R^2, 
\end{split}\]
where we just use the definition of $\rho$. 
Moreover, $E=\IE(X X^*)=\diag(1, \hdots,1)$ and so  $\|E\|=1$. 
Therefore, Proposition~\ref{prop:O} with 
$m=\left\lceil C_1 n\log n\right\rceil$ as above and
\[
t\,=\, 1+ C_3 \frac{2^\ell \log\bigl((\ell+1)n\bigr)}{v_\ell^2\,\log n}\,
\,\ge\, 2,
\] 
together with 
\[
\norm{\Gamma_\ell}^2 
\,:=\, \norm{\Gamma_\ell\colon \ell_2(I_\ell)\to\ell_2^m}^2 
\,=\, \norm{\Gamma_\ell^*\Gamma_\ell}
\,\le\, m + \norm{\Gamma_\ell^*\Gamma_\ell-mE},
\]
implies
\[
\IP\left(\norm{\Gamma_\ell}^2
\,\ge\, \frac{C_4\,n2^\ell\, \log\bigl((\ell+1)n\bigr)}{v_\ell^2}\right) 
\,\le\, \frac{4}{n^2(\ell+1)^2 \pi^2}
\]
for some constants $C_3,C_4>0$ 
(depending only on $C_1$).
Note that these probabilities are summable, and so 
a union bound shows 
\[
\IP\left(\exists\ell\in\IN_0\colon 
\norm{\Gamma_\ell}^2
\,\ge\, \frac{C_4\,n2^\ell\, \log\bigl((\ell+1)n\bigr)}{v_\ell^2} \right) 
\,\le\, \frac{1}{n^2}.
\]
We therefore obtain, with probability at least $1-\frac1{n^2}$
that
\begin{equation}\label{eq:N}
\begin{split}
\norm{N(f- P_n f)}_{\ell_2^m}
\,&\le\, C_4\, \sum_{\ell=0}^\infty
\sqrt{n2^\ell\,\log\left((\ell+1)n\right)}\, 
	\cdot \frac{\varepsilon_{n2^\ell}}{v_\ell} \\
\,&\le\, C_5\, \sqrt{\log n}\; \sum_{\ell=0}^\infty
\sqrt{n2^\ell\,\log\left(\ell+1\right)}\, 
	\cdot \frac{\varepsilon_{n2^\ell}}{v_\ell}.
\end{split}
\end{equation}
Clearly, the monotonicity of $(\varepsilon_n)$ gives
\[
 \sum_{k \ge n/2} \varepsilon_k^p \,\ge\, n (2^{\ell} -1/2) 
\varepsilon^p_{n2^\ell}
\]
and thus
\[
\varepsilon_{n2^\ell} 
\,\le\, C_6\, 2^{-\ell/p} \bigg( \frac1n \sum_{k \ge n/2} \varepsilon_k^p  \bigg)^{1/p}
\]
for all $\ell\ge 0$.
Inserting this in \eqref{eq:N} and using 
$v_\ell:=c_\delta\, 2^{-\delta\ell}$ 
for some $0<\delta <1/p-1/2$, 
with $c_\delta>0$ such that 
$\sum_{\ell\ge 0} v_\ell^2 = 1$,
yields
\[
 \norm{N(f- P_n f)}_{\ell_2^m}
\,\le\, C_7\, \sqrt{n\log n}\, \bigg( \frac1n \sum_{k \ge n/2} \varepsilon_k^p  \bigg)^{1/p}\sum_{\ell=0}^\infty
\sqrt{\log(\ell+1)}\, 2^{-(1/p-1/2-\delta)\ell}.
\]
Clearly, the latter series is finite. 
\end{proof}

\subsection{Kadison-Singer to reduce the number of points} 

We now employ the powerful solution to the Kadison-Singer 
problem due to Marcus, Spielman and Srivastava~\cite{MSS15} 
to show that we can reduce the number of points in our algorithm 
to $m\asymp n$, without losing the error bound. 
In detail, we need an equivalent version of the KS problem 
due to~\cite{We04}, which was already brought into a form that is very useful 
for us in~\cite{NSU20}. 
Note that the authors of \cite{NSU20} seem to be the first to 
use this approach in the context of sampling widths.
By this, they improved upon~\cite{KU19} and proved 
the result of Theorem~\ref{thm:main} 
for $F$ being the unit ball of a separable 
reproducing kernel Hilbert space. 
For applications of KS to the discretization of the $L_2$-norm, 
which was used to prove the result in~\cite{Tem20},
see e.g.~\cite{LT20,Tem17}. 
Here, we use a 
special case 
of~\cite[Theorem~2.3]{NSU20},  
see also~\cite[Lemma~2.2]{LT20} and~\cite[Lemma~2]{NOU16}.

\begin{prop}[{\cite[Theorem~2.3]{NSU20}}]
There exist constants $c_1,c_2,c_3>0$ such that, 
for all $u_1,\dots,u_m\in\IC^n$ such that 
$\|u_i\|_2^2\le \frac{2n}{m}$ 
for all $i=1,\dots,m$ and 
\[
\frac12\|w\|_2^2 \,\le\, \sum_{k=1}^m|\scalar{w}{u_i}|^2 \,\le\, \frac32\|w\|_2^2,
\qquad w \in \IC^n,
\]
there is a $J\subset\{1,\dots,m\}$ with 
$\#J\le c_1 n$ and 
\[
c_2\,\frac{n}{m}\,\|w\|_2^2 
\,\le\, \sum_{k\in J}|\scalar{w}{u_i}|^2 
\,\le\, c_3\,\frac{n}{m} \|w\|_2^2,
\qquad w \in \IC^n.
\]
\end{prop}

We now let 
$u_i=\frac{1}{\sqrt{m}}\left(\varrho(x_i)^{-1/2} b_j(x_i)\right)_{j\leq n} \in\IC^n$
with random $x_1,\hdots,x_m$ as in the previous section. 
We clearly have 
\[
 \sum_{k=1}^n|\scalar{w}{u_i}|^2=\frac1{m}\|Gw\|_2^2,  \qquad w \in \IC^n.
\]
Moreover, from \eqref{eq:norm-G}, \eqref{eq:bound-G} and \eqref{eq:N}
we see that $u_1,\hdots,u_m$ satisfy with high probability
the conditions of the proposition and
\[
 \sup\limits_{f\in \wtF} \norm{N(f- P_n f)}_{\ell_2^m}\,\le\, 
				c_4\, \sqrt{n\log n}\; \bigg( \frac1n \sum_{k \ge n/8} a_k(F_0, L_2)^p  \bigg)^{1/p}
\]
for some constant $c_4>0$, depending only on $p$.
The proposition yields
$J\subset\{1,\dots,m\}$ with 
$\#J\le c_1 n$ such that the matrix
\[
G_J \,:=\, \left(\varrho(x_i)^{-1/2} b_k(x_i)\right)_{i\in J,\, k\le n}
\]
satisfies
\[
c_2 \|w\|_2^2 \,\le\, \frac1{n}\|G_Jw\|_2^2 \,\le\, c_3 \|w\|_2^2,
\]
and hence 
\[
s_{\rm min}(G_J\colon \ell_2^n \to \ell_2(J))^2 \,\ge\, c_2 n.
\]
Moreover, we clearly have
\[
\sup\limits_{f\in \wtF} \norm{N_J(f- P_n f)}_{\ell_2(J)}
\,\le\, c_4\, \sqrt{n\log n}\; \bigg( \frac1n \sum_{k \ge n/8} a_k(F_0, L_2)^p  \bigg)^{1/p}
\]
for every $J\subset\{1,\dots,m\}$, 
where $N_J(f):=(\varrho(x_i)^{-1/2}f(x_i))_{i\in J}$.
Using Lemma~\ref{lem:bound}, we see that the algorithm
\[
A_{J,n}(f) \,:=\, \sum_{k=1}^n (G_J^+ N_J f)_k b_k
\,=\, \underset{g\in V_n}{\rm argmin}\, 
\sum_{i\in J} \frac{\vert g(x_i) - f(x_i) \vert^2}{\varrho(x_i)} 
\]
satisfies 
\[\begin{split}
e(A_{J,n}, F_0, L_2)^2 \,&\le\, 
\varepsilon_n^2 + s_{\rm min}(G_J\colon \ell_2^n \to \ell_2(J))^{-2}
	\cdot \sup_{f\in \wtF} \norm{N_J(f- P_n f)}_{\ell_2(J)}^2 \\
\,&\le\, c_5\, \log n \; \bigg( \frac1n \sum_{k \ge n/8} a_k(F_0, L_2)^p  \bigg)^{2/p}
\end{split}\]
for some $c_5>0$ that only depends on $p$.
Clearly, the same upper bound holds for $e_{c_1 n}(F_0, L_2)^2$
which proves Theorem~\ref{thm:main2}, and thereby Theorem~\ref{thm:main}.

\section{The limiting case $p=2$}
\label{sec:limit-case}

With our techniques, we were not able to prove a result for arbitrary 
spaces with $(a_n)\in\ell_2$, and it is not clear if this 
is possible, even if we allow weaker bounds.
However, the condition $(a_n)\in\ell_p$ with $p<2$
may be weakened to $((\log n)^q\, a_n )\in\ell_2$
for any $q>1$.
A second look at our proof
quickly reveals that we actually showed the upper bound
\begin{equation}\label{eq:general-bound}
 e_{cn}(F, L_2) \,\le\,
 C\, \sqrt{\frac{ \log n}{n}}\; \sum_{\ell=0}^\infty
\sqrt{n2^\ell \log\left(\ell+1\right)}\, 
	\cdot \frac{a_{n2^\ell}(F, L_2)}{v_\ell}
\end{equation}
for any sequence $(v_\ell)$ with $\sum_{\ell\ge 0}  v_\ell^2 = 1$,
all $n\ge 2$,
and with universal constants $c$ and $C$.
We simply 
skip the additional estimates after equation~\eqref{eq:N}.
Theorem~\ref{thm:main} is obtained from this estimate if we choose $ v_\ell$ with exponential decay.
The following is obtained if we choose $ v_\ell$ with polynomial decay.

\begin{thm}\label{thm:limit}
 Let $(D,\mathcal A,\mu)$ be a measure space and let 
 $F$ be a separable metric space of complex-valued functions on $D$
 that is continuously embedded into $L_2(D,\mathcal A,\mu)$
 such that function evaluation is continuous on $F$. 
 Then there are universal constants $c,C\in \IN$ such that,
 for all $n\ge 3$,
 \[
 e_{cn}(F, L_2) \,\le\, 
 C\cdot \left(\frac{\log n}{n}\,
 \sum_{k\ge n} (\log k)^2 (\log \log k)^5 a_k(F, L_2)^2 \right)^{1/2}.
 \]
\end{thm}

If we apply this estimate to a sequence
satisfying $a_n(F, L_2) \lesssim n^{-1/2} \log^\beta (n+1)$
for some 
$\beta <-3/2$,
we obtain for any $\varepsilon>0$ that
\[
  e_n(F, L_2) \,\lesssim \, n^{-1/2} \log^{\beta+2+\varepsilon} (n+1).
\]

\begin{proof}
We apply Cauchy Schwarz to \eqref{eq:general-bound},
choosing the square summable sequence $(\ell^{1/2}\log(\ell+1))^{-1}$
as the first factor, and obtain
\begin{align*}
 e_{2cn}(F, L_2)^2 \,&\le\,
 \widetilde C\, \frac{\log n}{n}\, \sum_{\ell=0}^\infty 
	n2^\ell \cdot \ell\, \log^3\left(\ell+1\right)\, 
	\cdot \frac{a_{n2^{\ell+1}}(F, L_2)^2}{ v_\ell^2}  \\
	&\le\,
	\widetilde C\, \frac{\log n}{n}\, \sum_{\ell=0}^\infty 
	\frac{\ell \log^3\left(\ell+1\right)}{ v_\ell^2}  \sum_{k \in I_\ell} a_k(F, L_2)^2.
\end{align*}
Recall $I_\ell=\{ n2^\ell + 1, \hdots, n2^{\ell+1} \}$.
The statement is obtained by setting $ v_\ell^{-2} = \tilde c \ell \log^2(\ell+1)$,
noting that $\ell \le \log(k)$ for $k\in I_\ell$.
\end{proof}

\medskip

\subsection*{Acknowledgement}
We thank Aicke Hinrichs, Erich Novak, Winfried Sickel, Vladimir Temlyakov and Tino Ullrich
for various helpful comments.
In particular, we thank Erich Novak
for providing us with Example~\ref{counterexample},
and Tino Ullrich for pointing out the present (more elegant) formulation 
of the upper bound in Theorem~\ref{thm:main}
and for coming up with the example of Korobov classes.
David Krieg is supported by the Austrian Science Fund (FWF) Project F5513-N26, 
which is a part of the Special Research Program \emph{Quasi-Monte Carlo Methods:~Theory and Applications}.

\bigskip

%%%%%%%%%%%

\end{document}